\documentclass[a4paper,11pt]{article}
\usepackage[T1]{fontenc}
\usepackage[utf8]{inputenc}
\usepackage[english]{babel}
\usepackage[margin=0.8in]{geometry}
\usepackage{microtype}
\usepackage{braket}
\usepackage{amsmath}
\usepackage{amssymb}
\usepackage{amsthm, aliascnt}
\usepackage{mathtools}
\usepackage{mathrsfs}
\usepackage{xcolor}
\usepackage[hyphens]{url}
\usepackage{hyperref}
\usepackage[hyphenbreaks]{breakurl}
\usepackage{enumitem}
\usepackage{tikz}
\usepackage[maxbibnames=99]{biblatex}
\usepackage{csquotes}
\usetikzlibrary{patterns}
\usepackage{tikz}
\usepackage{refcheck}
\usepackage{esint}

\hypersetup{
                colorlinks=true,
                linkcolor={magenta},
                citecolor={cyan},
                breaklinks=true}

\theoremstyle{plain}
\newtheorem{teor}{Theorem}
\numberwithin{teor}{section}
\numberwithin{equation}{section}
\theoremstyle{definition}
\newaliascnt{defi}{teor}
\newtheorem{defi}[defi]{Definition}
\aliascntresetthe{defi}

\theoremstyle{plain}
\newaliascnt{lemma}{teor}%
\newtheorem{lemma}[lemma]{Lemma}
\aliascntresetthe{lemma}
\theoremstyle{plain}
\newaliascnt{prop}{teor}%
\newtheorem{prop}[prop]{Proposition}
\aliascntresetthe{prop}
\theoremstyle{plain}
\newaliascnt{cor}{teor}%
\newtheorem{cor}[cor]{Corollary}
\aliascntresetthe{cor}
\theoremstyle{definition}
\newaliascnt{ex}{teor}%

\aliascntresetthe{ex}
\theoremstyle{definition}
\newaliascnt{oss}{teor}%
\newtheorem{oss}[oss]{Remark}
\aliascntresetthe{oss}
\theoremstyle{plain}

\DeclarePairedDelimiter{\abs}{\lvert}{\rvert}
\DeclarePairedDelimiter{\norma}{\lVert}{\rVert}
\DeclareMathOperator{\seg}{seg}

\DeclareMathOperator{\sn}{sn}
\DeclareMathOperator{\cn}{cn}

\DeclareMathOperator{\arcsinh}{arcsinh}

\DeclareMathOperator{\Lip}{Lip}
\DeclareMathOperator{\Mink}{Mink}
\newcommand{\R}{\mathbb{R}}
\newcommand{\Sp}{\mathbb{S}}

\newcommand{\Hn}{\mathcal{H}^{n-1}}

\newcommand{\ka}{\kappa}
\newcommand{\eps}{\varepsilon}

\DeclareMathOperator{\cut}{Cut}
\makeatletter
\newcommand{\leqnomode}{\tagsleft@true\let\veqno\@@leqno}
\newcommand{\reqnomode}{\tagsleft@false\let\veqno\@@eqno}

\newcommand{\scalar}[2]{\langle #1,#2 \rangle }
\newcommand\restr[2]{{%
  \left.\kern-\nulldelimiterspace %
  #1 %
  \vphantom{\big|} %
  \right|_{#2} %
  }}

\makeatletter
\newcommand{\refcheckize}[1]{%
  \expandafter\let\csname @@\string#1\endcsname#1%
  \expandafter\DeclareRobustCommand\csname relax\string#1\endcsname[1]{%
    \csname @@\string#1\endcsname{##1}\wrtusdrf{##1}}%
  \expandafter\let\expandafter#1\csname relax\string#1\endcsname
}
\makeatother

\refcheckize{\autoref}

\definecolor{paolo}{rgb}{0.09, 0.45, 0.27}

\DeclareMathOperator{\id}{Id}
\title{A spectral isoperimetric inequality on the $n$-sphere for the Robin-Laplacian with negative boundary parameter}
\author{P. Acampora, A. Celentano, E. Cristoforoni, C. Nitsch, C. Trombetti}
\date{}

\newcommand{\Addresses}{{%
 \bigskip 
 \footnotesize 
 
 \textsc{Dipartimento di Matematica e Applicazioni ``R. Caccioppoli'', Universit\`a degli studi di Napoli Federico II, Via Cintia, Complesso Universitario Monte S. Angelo, 80126 Napoli, Italy.}\par\nopagebreak 
 
 \medskip 
 
 \textit{E-mail address}, P.~Acampora: \texttt{paolo.acampora@unina.it} 

 \medskip

 \textit{E-mail address}, A.~Celentano: 
 \texttt{antonio.celentano2@unina.it}
  
 \medskip 
 
 \textit{E-mail address}, C.~Nitsch: \texttt{c.nitsch@unina.it}

  \medskip 
 
 \textit{E-mail address}, C.~Trombetti: \texttt{cristina@unina.it} 
 
 \medskip 
 
\textsc{Mathematical and Physical Sciences for Advanced Materials and Technologies, Scuola Superiore Meridionale, Largo San Marcellino 10, 80126, Napoli, Italy.}\par\nopagebreak 
 
 \medskip 
 
 \textit{E-mail address}, E.~Cristoforoni: \texttt{emanuele.cristoforoni@unina.it} 
}}

\begin{document}
\renewcommand*{\sectionautorefname}{Section}
\renewcommand*{\subsectionautorefname}{Subsection}

\maketitle
\begin{abstract}
For every given $\beta<0$, we study the problem of maximizing the first Robin eigenvalue of the Laplacian $\lambda_\beta(\Omega)$ among convex (not necessarily smooth) sets $\Omega\subset\mathbb{S}^{n}$ with fixed perimeter. In particular, denoting by $\sigma_n$ the perimeter of the $n$-dimensional hemisphere, we show that for fixed perimeters $P<\sigma_n$, geodesic balls maximize the eigenvalue. Moreover, we prove a quantitative stability result for this isoperimetric inequality in terms of volume difference between $\Omega$ and the ball $D$ of the same perimeter. 


\end{abstract}
\section{Introduction}

Let $(M,g)$ be a complete Riemannian $n$-manifold, and let $\Omega\subset M$ be a bounded open set with smooth boundary. For every $\beta\in\R$ consider the Robin-Laplacian eigenvalue problem on $\Omega$, that is
\begin{equation}\label{robinproblem}\begin{cases}
    -\Delta u=\lambda u &\text{in }\Omega,\\[5 pt]
    \dfrac{\partial u}{\partial \nu}+\beta u=0 &\text{on }\partial\Omega,
\end{cases}\end{equation}
where $\Delta$ is the Laplace-Beltrami operator on $M$ and $\nu$ is the unit outer normal to the boundary of $\Omega$. \eqref{robinproblem} admits an increasing sequence of eigenvalues diverging to infinity. Moreover, if $\Omega$ is connected, any first eigenfunction has a sign, so that, by linearity, the first eigenvalue $\lambda_\beta(\Omega)$ is simple (see for example \cite{LW20}).

Let $\lambda_\beta(\Omega)$ be the smallest eigenvalue for \eqref{robinproblem}, then the following variational characterization holds 
\begin{equation}\label{robin}\lambda_\beta(\Omega)=\inf_{v\in H^1(\Omega)} \dfrac{\displaystyle \int_\Omega \abs{\nabla v}^2\,d\mu+\beta\int_{\partial\Omega}v^2\,d\Hn}{\displaystyle\int_\Omega v^2\,d\mu}.\end{equation}
Any minimizer $u$ of \eqref{robin} is a weak solution to \eqref{robinproblem} for $\lambda=\lambda_\beta(\Omega)$, that is
\begin{equation*}
\int_\Omega g(\nabla u,\nabla \varphi)\,d\mu+\beta\int_{\partial\Omega}u\varphi\,d\Hn=\lambda_\beta(\Omega)\int_\Omega u\varphi\,d\mu,
\end{equation*}
for every $\varphi\in H^1(\Omega)$.
An immediate consequence of the variational characterization \eqref{robin} is the fact that the function 
\[\beta\in\R\mapsto\lambda_\beta(\Omega)\in\R\]
is increasing. In particular, for $\beta=0$, the Robin boundary condition coincides with the Neumann one and $\lambda_0(\Omega)=0$ with constant eigenfunctions. Therefore, the first Robin eigenvalue is positive for $\beta>0$ and negative for $\beta<0$. 

\medskip
Comparison theorems for the first Robin eigenvalue are widely studied in the literature. The first example of such theorems is probably the one due to Bossel in \cite{B86}: this result generalizes the Faber-Krahn inequality for the first Robin eigenvalue with $\beta>0$ in the class of bounded open sets of the Euclidean plane $\R^2$. Namely, let $\Omega\subset\R^2$ be a bounded open set and let $B\subset\R^2$ be a ball having the same area, then
\begin{equation}\label{bossel-daners}\lambda_\beta(\Omega)\ge\lambda_\beta(B).\end{equation}
Daners generalized the previous result in \cite{D06} for bounded open subsets of the Euclidean space $\R^n$.

In the context of Riemannian manifolds, one usually compares the Robin eigenvalue of a bounded domain $\Omega$ in a complete manifold $M$ with the one of a geodesic ball in an appropriate simply connected space form. In particular, Chen, Cheng, and Li in \cite{CCL22} proved a Bossel-Daners inequality \eqref{bossel-daners} for bounded domains of a manifold $M$, where either $M$ is the hyperbolic space or it is a compact manifold whose Ricci curvature tensor satisfies a positive lower bound. As proved by Chen, li, and Wei in \cite{CLW23}, the inequality still holds in the case in which $M$ is a complete, non-compact, manifold whose Ricci tensor is non-negative.\medskip

In the case $\beta<0$, Bareket in \cite{B77} famously conjectured that among all Lipschitz sets of a given area in the Euclidian plane, the ball maximizes the first Robin eigenvalue. Freitas and Krej{\v{c}}i{\v{r}}{\'{\i}}k in \cite{FK15} disproved the conjecture: they proved, via an asymptotic expansion, that, for $\abs{\beta}$ sufficiently large, the first Robin eigenvalue of an annulus is larger than the one of the ball having the same measure. At the same time, they proved that for smooth bounded subsets of the Euclidean plane, the conjecture holds true provided that $\beta$ is sufficiently close to $0$. However, fixing the perimeter leads to other interesting comparisons. Indeed, Antunes, Freitas, and Krej{\v{c}}i{\v{r}}{\'{\i}}k in \cite{AFK16} proved a comparison theorem for the first Robin eigenvalue, with $\beta<0$, under a perimeter constraint. Namely, let $\Omega$ be a bounded open set with $C^2$ boundary in the Euclidean plane and let $B\subset\R^2$ be a ball having the same perimeter, then
\begin{equation}\label{robneg}\lambda_\beta(\Omega)\le\lambda_\beta(B).\end{equation}
Bucur et al. in \cite{BFNT19} proved that the inequality \eqref{robneg} holds true in any dimension provided that we restrict the class of admissible sets to the one of the convex sets, or, more in general, the inequality holds for any Lipschitz set which can be written as $\Omega\setminus K$, where $\Omega$ is open and convex and $K$ is a closed set in $\Omega$. Vikulova in \cite{V22} proved the result in the Euclidean space $\R^3$ for bounded convex sets or connected axiconvex sets whose boundary is diffeomorphic to the sphere. 

In the context of Riemannian manifolds, Khalile and Lotoreichik in \cite{KL22} proved the following. Let $\Omega$ be a compact, two-dimensional, simply connected Riemannian manifold with $C^2$ boundary and with Gauss curvature bounded from above by a non-negative constant $\ka_0$, and let $B$ be a geodesic disc in the simply connected space form of Gauss curvature $\ka_0$ with the same perimeter as $\Omega$. Then, for every $\beta<0$, inequality \eqref{robneg} holds.\medskip

Finally, in Riemannian manifolds, other comparison theorems for the first Robin eigenvalue and domain monotonicity properties have been proved by Savo in \cite{S20} and by Li and Wang in \cite{LW20}.

\medskip
 The main objective of this paper is to adapt the techniques of \cite{BFNT19} to prove the following theorem. Note that we denote by $\Hn$ the Hausdorff measure, and we refer to \autoref{defi: convex} for the definition of strong convexity.

\begin{teor}
    \label{teorema1}
    Let $\Omega\subset\Sp^n$ be an open set such that $\bar{\Omega}$ is strongly convex, and let $D$ be a strongly convex geodesic ball with 
    \[
    \Hn(\partial\Omega)=\Hn(\partial D).
    \]
    Then
    \begin{equation}
    \label{eq: negrobin}
        \lambda_\beta(\Omega)\le\lambda_\beta(D),
    \end{equation}
    and the equality holds if and only if, up to a translation, $\Omega=D$.
\end{teor}
Notice that, thanks to \cite[Theorem 5]{S20}, we have that the eigenvalue is increasing with respect to the inclusion among balls, so that \eqref{eq: negrobin} also holds true replacing $D$ with $\Sp^n\setminus D$.

The proof relies on the use of the method of parallel coordinates (see \cite{PW61} and \cite{CG01}) to construct a suitable test function on convex subsets of the sphere $\Sp^n$. Nevertheless, the main difficulty here was to recover classical results about convex sets on the sphere. In particular, the main ingredients of the proof are:
\begin{enumerate}[label=(\roman*)]
    \item convexity properties of inner and outer parallel sets;
    \item monotonicity of perimeters with respect to the inclusion for convex sets;
    \item Steiner's formulae;
    \item Alexandrov-Fenchel inequality for the mean curvature.
\end{enumerate}
To infer convexity properties of inner parallel and outer parallel sets (see \autoref{defi: parallelsets}) we need some convexity property of the distance function provided by Bangert in his paper \cite{B78} (see \autoref{teor: convexdist}). The monotonicity of the perimeter has been proved by Bangert in \cite{B81} (see \autoref{teor: monotonicity}). The Steiner's formulae have been extended to $C^2$ convex sets of the sphere by Allendoerfer in \cite{A48} and the Alexandrov-Fenchel inequality has been recently extended to $C^2$ convex sets of the sphere by Makowski and Scheuer in \cite{MS16}. However, our result in \autoref{teorema1} only requires the set to be convex: to avoid the constraint on the regularity of the boundary, we recover a general theory for Steiner's formulae and curvature measures introduced by Federer in \cite{F59} in $\R^n$ and successively generalized to simply connected space forms by Kohlmann in \cite{K91}. We are then able to generalize Alexandrov-Fenchel inequalities for general convex sets (see \autoref{cor: alexfench}) by approximating convex sets with smooth convex sets, using a result that has been proved by Bangert in \cite{B78} (see also \autoref{teor: approx}).

In addition, we are also able to adapt the techniques in \cite{AGM22} to prove a stability result resumed in the following 
\begin{teor}
    \label{teorema2}
    Let $\Omega\subset\Sp^n$ be an open set such that $\overline{\Omega}$ is strongly convex, and let $D$ be a strongly convex geodesic ball such that 
\[
\Hn(\partial\Omega)=\Hn(\partial D).
\]
For every $\beta<0$, let $u$ be an eigenfunction relative to $\lambda_\beta(D)$, and let 
\[u_m=\min_{p\in\overline{D}} u(p).\]
Then,
\begin{equation}\label{quant}\dfrac{\lambda_\beta(D)-\lambda_\beta(\Omega)}{\abs{\lambda_\beta(\Omega)}}\ge\dfrac{u_m^2}{\norma{u}_{L^2(D)}^2}(\abs{D}-\abs{\Omega}).\end{equation}
\end{teor}
\medskip
The paper is organized as follows. In \autoref{notations} we give introductory notions and classical tools of Riemannian manifolds and integration theory. In \autoref{convexity} we give classical results and definitions about convexity in Riemannian manifolds, with special attention to the convexity of the inner parallel sets and convex approximation. In \autoref{curvatures} we give the definition of curvature measures, and we state the Steiner formula in $\Sp^n$ and the Alexandrov-Fenchel inequality. In \autoref{proof1} we prove \autoref{teorema1} and \autoref{teorema2}. Finally, in \autoref{remarks} we discuss the limits of the proof in the hyperbolic space.

\section{Notation and tools}
\label{notations}
In the following, given a smooth, orientable Riemannian $n$-manifold $(M,g)$, we will denote by $d$ the Riemannian distance 
\[
    d(p,q)=\min_{\substack{\gamma\in C^\infty((0,1);M) \\ \gamma(0)=p \\ \gamma(1)=q}}\:\int_0^1 g(\gamma'(t),\gamma'(t))\,dt
\]
induced by $g$; by $d\mu$ its volume form which is expressed locally in coordinates as
\[
    d\mu=\sqrt{\abs{\det(g_{ij})}}dx_1\dots dx_n;
\]
and we will denote by $\abs{\cdot}$ the classical Riemannian volume
\[
\abs{E}=\int_E\,d\mu.
\]
We let $TM$ denote the tangent bundle on $M$, by $\Gamma(TM)$ the sections of the bundle, namely the space of vector fields, and by $T_pM$ the tangent space at $p$. We also recall that for every $(v,p)\in TM$ a \emph{geodesic} starting from $p$ with velocity $v$ is the unique curve $\gamma=\gamma_{p,v}$ such that $\gamma(0)=p$, and $\gamma'(0)=v$, and such that it solves the system of equations written in local coordinates as (using the Einstein notation on repeated indices) 
\[
\gamma_i''(t)+\Gamma^i_{jk}(\gamma(t))\,\gamma'_j(t)\,\gamma'_k(t)=0, \qquad \qquad i=1,\dots,n
\]
with $\Gamma^i_{jk}$ representing the Christoffel symbols of the metric $g$. When $M$ is complete, we can extend the geodesics $\gamma_{p,v}\in C^\infty(\R;M)$, and we denote by
\[
\exp:TM\to M \qquad \exp_p: T_pM\to M
\]
the exponential map defined as
\begin{equation}
\label{eq: expdef}
\exp(p,v)=\exp_p(v)=\gamma_{p,v}(1).
\end{equation}
For every $p\in M$ we will denote the cutlocus of $p$ in $M$ as 
\[
\cut(p)=\exp_p(\partial\seg(p)),
\]
where
\[
\seg(p)=\Set{v\in T_p M | \gamma_{p,v} \text{ minimizes the distance $d(p,\gamma_{p,v}(1))$}}.
\]

We will denote by $\mathcal{H}^k$ the Hausdorff measure relative to the Riemannian distance on $M$. When necessary, we will denote the Hausdorff measure by $\mathcal{H}^k_g$ to highlight the dependence on the metric $g$. We refer to \cite[Section IV]{C01} for basic properties on this topic in the Riemannian setting. We denote by $\sigma_n$ the $(n-1)$-dimensional measure of the boundary of a hemisphere in the sphere $\Sp^n$ of sectional curvature $1$. Moreover, we will denote by $\langle\cdot,\cdot\rangle$ the canonical scalar product in $\R^n$. 

\subsection{General notions}
In the following, we will need some approximation argument. Hence, we define the Hausdorff distance of sets. Let us recall that given a closed set $\Omega\subset M$ the distance from $\Omega$ is defined as 
\[
    d(p,\Omega)=\inf_{q\in \Omega}d(p,q).
\]
\begin{defi}
\label{defi: parallelsets}
    Let $M$ be a Riemannian manifold, and let $K\subset M$ be a compact set. For every $t\ge 0$, we define the \emph{inner parallel set}
    \[
        (K)_t=\Set{p\in K | d(p,\partial K)\ge t},
    \]
    and the \emph{outer parallel} set
    \[
        (K)^t =\Set{p\in M | d(p,K)\le t}.
    \]
\end{defi} 
\begin{defi}[Hausdorff distance]
Let $M$ be a Riemannian manifold, and let $K_1,K_2\subset M$ be two compact sets. We define the Hausdorff distance as
\[
d^H(K_1,K_2)=\inf\Set{t\ge 0 | \begin{gathered}
    K_1\subset (K_2)^t \\
    K_2\subset (K_1)^t
\end{gathered}}
\]
\end{defi}
We refer to \cite[\S 2]{W76} for the following definitions.
\begin{defi}
    Let $M$ be a Riemannian $n$-manifold, and let $\Sigma\subset M$. We say that $\Sigma$ is a \emph{strongly Lipschitz submanifold} of $M$ of dimension $k$ if for every $p\in \Sigma$ there exist a $C^1$ chart $(U,\varphi)$ in $M$ around $p$, an open set $U'\subset\R^{k}$, and a Lipschitz function $f:U'\to \R^{n-k}$ such that
    \[
        \varphi(\Sigma\cap U)=\Set{(x,f(x))\in \varphi(U) | x\in U'}.
    \]
\end{defi}
\begin{defi}
    Let $M$ be a Riemannian $n$-manifold, and let $\Omega\subset M$. We say that $\Omega$ \emph{has strongly Lipschitz boundary} if $\Omega=\overline{\mathring{\Omega}}$, and $\partial \Omega$ is a strongly Lipschitz submanifold of $M$ of dimension $n-1$.
\end{defi}
\begin{defi}
    Let $X, Y$ be two metric spaces. We say that a homeomorphism
    \[
        f:X\to Y
    \]
    is \emph{locally bi-Lipschitz} if both $f$ and $f^{-1}$ are locally Lipschitz.
\end{defi} 
\begin{defi}
\label{defi: secfundform}
Let $(M,g)$ be a Riemannian $n$-manifold, and let $\Sigma$ be a $C^2$ oriented, embedded $(n-1)$-submanifold of $M$. We define the \emph{second fundamental form} $h$ of $\Sigma$ in $M$ as the 2-form such that for every $X,Y\in\Gamma(T\Sigma)$
\[
h(X,Y)=g(X,\nabla_Y\nu),
\]
where $\nabla$ is the Levi-Civita connection of $M$, and $\nu$ is the normal to $\Sigma$. 
\end{defi}
\begin{prop}
    Let $M$ and $\Sigma$ as in \autoref{defi: secfundform}. Then:
    \begin{enumerate}[label=(\roman*)]
        \item $h$ is symmetric, namely
        \[
            h(X,Y)=h(Y,X) \qquad \forall X,Y\in\Gamma(T\Sigma);
        \]
        \item for every $\sigma\in\Sigma$ there exist $n-1$ eigenvalues $k_1(\sigma)\le\dots\le k_{n-1}(\sigma)$ of $h$ and we say that $k_i$ are the \emph{principal curvatures of $\Sigma$}.
    \end{enumerate}
\end{prop}
\begin{defi}
Let $M$ be a Riemannian $n$-manifold, let $\Sigma$ be a $C^2$ oriented, compact, embedded $(n-1)$-submanifold of $M$. For every $p\in\Sigma$ and for every $1\le j \le n-1$, we denote by
\[
H_j(p)=\sum_{1\le i_1<\dots<i_j\le n-1} k_{i_1}(p)\dots k_{i_j}(p)
\]
the $j$-th homogeneous symmetric form of the principal curvatures, and
\[
H_0(p)=1.
\]
In particular, we say that $H_1(p)$ is the \emph{mean curvature} of $\Sigma$ in $p$.

\medskip
We now state the coarea and area formula.
\begin{defi}
   Let $V$ be a normed vector space of dimension $n$. For every $r=1,\dots,n$ we denote by $\bigwedge\nolimits_r V$ the space of alternating $r$-forms on the dual $V^*$. 
\end{defi}
If $V=T_p M$ is a tangent space for a Riemannian $n$-manifold $M$ at a point $p$, for every $r\le n$ we use the notation 
\[
\bigwedge\nolimits_r M_p:=\bigwedge\nolimits_r T_pM
\]
to denote the inner product of $r$ copies of $T_pM$.
\begin{defi}
        Let $(M,g)$ be a Riemannian $n$-manifold of class $C^1$, let $(N,h)$ be a Riemannian $k$-manifold of class $C^1$, let 
        \[
            r=\min\{n,k\},
        \]
        and let $f:M\to N$ be a map such that $f$ is differentiable in $p\in M$. We define the natural extension of $df_p$ to $\bigwedge_r M_p$ as the linear map 
        \[
        \wedge_r df_p:\bigwedge\nolimits_r M_p\to\bigwedge\nolimits_r N_{f(p)}
        \]
        such that for every $v_1,\dots,v_r\in T_pM$
        \[
           \wedge_{r}df_p(v_1\wedge\dots\wedge v_{r})=df_p(v_1)\wedge\dots df_p(v_{r}).
        \]
        We define the \emph{jacobian} of $f$ as
        \[
            Jf(p)=\norma{\wedge_{r}df_p}, 
        \]
        where the norm $\norma{\cdot}$ denotes the operatorial norm in the space of linear applications $\mathcal{L}(\bigwedge_rM_p,\bigwedge_r N_{f(p)})$ with the respective norms $\norma{\cdot}_{g,p}$ and $\norma{\cdot}_{h,f(p)}$.
\end{defi}  
For the proof of the following theorem, we refer to \cite[Theorem 3.1]{F59}
\begin{teor}[Coarea Formula]
    \label{teor: coarea}
    Let $(M,g)$ be a Riemannian $n$-manifold, let $(N,h)$ be a Riemannian $k$-manifold with $n\ge k$, and let $f:M\to N$ be a Lipschitz map. Then $f$ is $\mathcal{H}^n$-a.e. differentiable and for every $\mathcal{H}^n$-integrable function $\varphi:M\to \R$ we have
    \[
        \int_M\varphi(x)\,J\!f(x)\,d\mathcal{H}^n(x)=\int_N\int_{f^{-1}(y)}\varphi(z)\,d\mathcal{H}^{n-k}(z)\,d\mathcal{H}^k(y).
    \]
\end{teor}
For the following theorem we refer to \cite[Theorem 3.2.5, Remark 3.2.46]{F96}.
\begin{teor}[Area Formula]
    \label{teor: area}
    Let $(M,g)$ be a Riemannian $n$-manifold, let $(N,h)$ be a Riemannian $k$-manifold with $n\le k$, and let $f:M\to N$ be a Lipschitz map. Then $f$ is $\mathcal{H}^n$-a.e. differentiable and for every $\mathcal{H}^n$-measurable function $\varphi:M\to \R$ and we have
    \[
        \int_M\varphi(x)\,J\!f(x)\,d\mathcal{H}^n(x)=\int_N\int_{f^{-1}(y)}\varphi(z)\,d\mathcal{H}^{0}(z)\,d\mathcal{H}^k(y).
    \]
\end{teor}
\end{defi}

\subsection{Convexity in Riemannian manifolds}
\label{convexity}
In this section, we aim to give a general overview of convexity in Riemannian manifolds, and then we will study properties of convex sets in the specific case of the sphere $\Sp^n$. In order to give some convexity definitions in the Riemannian setting, we introduce the notions of supporting cone and normal cone. (We recall the definition of the exponential map in \eqref{eq: expdef}.)
\begin{defi}
Let $M$ be a Riemannian manifold and $C\subset M$ with non-empty interior. For every $p\in\partial C$ we define the \emph{(local) supporting cone} of $C$ in $p$ as
\[
\mathcal{C}_C(p)=\Set{\xi\in T_pM | \exists\eps>0 : \exp_p(t\xi)\in\mathring{C} \quad \forall t\in(0,\eps)},
\]
and the \emph{(internal) normal cone} as its dual cone 
\[
\mathcal{C}_C(p)^*=\Set{\nu \in T_pM | \scalar{\nu}{\xi}\ge 0 \quad \forall\xi\in\mathcal{C}_C(p)}.
\]
\end{defi}
Then, recalling that we use the notation $\overline{pq}$ to denote the minimal geodesic connecting $p$ and $q$ is unique in $M$,  we give the following definitions

\begin{defi}
\label{defi: convex}
    Let $M$ be a Riemannian manifold, and let $C_1,C_2\subset M$. We say that:
    \begin{enumerate}[label=(\alph*)]
    \item $C_1$ is \emph{weakly convex} if for every $p,q\in C_1$ there exists a minimal geodesic $\gamma:[a,b]\to M$ connecting $p$ and $q$ contained in $C_1$;
    \item $C_1$ is \emph{strongly convex} if for every $p,q\in C_1$ there exists a unique minimal geodesic $\overline{pq}$ connecting $p$ and $q$ in $M$, and $\overline{pq}\subseteq C_1$;
    \item $C_1$ is \emph{locally convex} if for every $p\in \bar{C_1}$ there exists $\eps>0$ and a metric ball $B_\eps(p)$ such that $C_1\cap B_\eps(p)$ is strongly convex;
    \item $C_1$ is \emph{locally strictly convex} if there exists a $\delta>0$ such that for every point $p\in \partial C_1$ and for every $\nu\in \mathcal{C_1}_C(p)^*$ the following holds: there exists an hypersurface $H$ orthogonal to $\nu$ in $p$ such that $H\cap C_1=\{p\}$ and its second fundamental form in $p$ with respect to $\nu$ has eigenvalues greater than $\delta$;
    \item $C_1$ is \emph{totally convex} in $C_2$ if $C_1\subseteq\mathring{C_2}$ and for every $p,q\in C_1$ and every geodesic
    \[
    \gamma:[a,b]\to C_2
    \]
    connecting $p$ and $q$ inside $C_2$ we have $\gamma([a,b])\subseteq C_1$.
    \end{enumerate}
\end{defi}
\noindent We refer to \cite{CG72} for definitions (a)-(c), to \cite{B78} for definition (d), and to \cite{B81} for definition (e).

We now give some useful properties about convex sets in the sphere. 
\begin{oss}
\label{oss: convexinhemi}
Recall that the definition of strong convexity is actually imposing some geometric constraint on the set $C$. For instance, on the sphere $\Sp^n$ we have that if $C\subseteq\Sp^n$ is a closed strongly convex set, then $C$ is contained in an open hemisphere. Indeed, let $C\subseteq\Sp^n$ be a closed strongly convex set. By definition of strong convexity, we have that if $p\in C$ then necessarily the antipodal point $-p\notin C$. Therefore, we can find a plane separating $C$ and its antipodal set $-C$: indeed, 
\[
\Omega^+:=\Set{tx\in\R^{n+1} | \begin{gathered}
    t>0, \\
    x\in C
\end{gathered}}
\]
and 
\[
\Omega^-:=\Set{tx\in\R^{n+1} | \begin{gathered}
    t>0, \\
    x\in -C
\end{gathered}}
\]
are two disjoint convex cones in $\R^{n+1}$, and they can be separated by a plane passing through the origin. This in particular implies that $C$ is contained in a hemisphere.
\end{oss}
\begin{oss}
\label{oss: weakinhemi}
    If $C\subset\Sp^n$ is weakly convex and it is contained in a hemisphere, then it is strongly convex, since for every couple of points $p,q\in C$ there exists a unique minimal geodesic connecting them.
\end{oss}
\begin{oss}
\label{oss: strongtototal}
    Notice that if $C_1,C_2\subset\Sp^n$ are two strongly convex sets such that $C_1\subseteq C_2$, then $C_1$ is totally convex in $C_2$. Indeed, since $C_2$ is contained in a hemisphere, then for every couple of points $p,q\in C_1$, the unique minimal geodesic $\overline{pq}$ connecting $p$ and $q$ is also the unique geodesic connecting $p$ and $q$ contained $C_2$.
\end{oss}

Notice that the definition of totally convex set becomes trivial when $M$ is a compact manifold and we take $C_2=M$. See for instance \cite[Corollary 1]{B81} for the following
\begin{prop}
    Let $M$ be a compact connected Riemannian manifold, and let $C\subseteq M$ be a totally convex set in $M$. Then $C=M$.
\end{prop}

\begin{oss}
    Notice that if $C$ is strongly convex, then it is connected and locally convex. 

    Notice also that if $C_1$ is strongly convex and $C_1\subset C_2$ is totally convex in $C_2$, then $C_1$ is strongly convex.
\end{oss}
In $\Sp^n$, open, connected, locally convex sets contained in a hemisphere have to be strongly convex. Indeed, we can characterize weak convexity with some geometric properties of the boundary. Let us introduce the notion of supporting element (see \cite{CG72, A78}). 
\begin{defi}
    Let $M$ be a Riemannian manifold, and let $C\subseteq M$ be an open set. Let $p\in\partial C$, and for some $\nu\in T_pM$ define
    \[
        H_p=\Set{\xi \in T_p M | \scalar{\nu}{\xi}< 0}.
    \]
    We say that:
    \begin{enumerate}[label=(\roman*)]
     \item  the half-space $H_p$ is a \emph{supporting element} for $C$ in $p$ if for every $q\in \mathring{C}$ and for every minimal geodesic 
     \[
        \gamma:[0,1]\to M
    \]
    such that $\gamma(0)=p$  and $\gamma(1)=q$, we have $\gamma'(0)\in H_p$;
    \item the half-space $H_p$ is a \emph{locally supporting element} for $C$ in $p$ if there exists a neighbourhood $U$ of $p$ such that $H_p$ is a supporting element for $U\cap C$ in $p$.
    \end{enumerate}
\end{defi}
Let $M$ be a Riemannian manifold, and for every $p\in M$, let $\cut(p)$ be the cut-locus of $p$. We refer to \cite[Proposition 2]{A78} for the following result.
\begin{prop}
\label{prop: weakchar}
    Let $M$ be a Riemannian manifold, and let $C\subset M$ be connected and open. The set $C$ is weakly convex if and only if for every point $p\in \partial C$ there exists a locally supporting element and $C\setminus\cut(p)$ is connected.
\end{prop}

We also have that a locally supporting element always exists for open, locally convex sets. Indeed, Cheeger and Gromoll in \cite[Theorem 1.6, Lemma 1.7]{CG72} proved a result summarized in \autoref{teor: supporting} (see also the comments between Lemma 1.7 and Proposition 1.8); notice that Cheeger and Gromoll work with closed sets, but if $C$ is locally convex, then also $\bar{C}$ is a locally convex set, and $\partial C=\partial\bar{C}$. On the other hand, by definition, a supporting element for $\bar{C}$ is also a supporting element for $C$.
\begin{teor}
\label{teor: supporting}
    Let $M$ be a Riemannian manifold of dimension $n$, and let $C\subseteq M$ be a non-empty, open, locally convex set. Then $\partial C$ is an embedded $(n-1)$-dimensional topological submanifold of $M$, and it has a supporting element in every point $p\in\partial C$.
\end{teor}
Joining \autoref{prop: weakchar} and \autoref{teor: supporting}, we get on the sphere $\Sp^n$ the following.
\begin{prop}
\label{prop: localtostrong}
    Let $C\subset\Sp^n$ be a closed, connected, locally convex set contained in an open hemisphere. Then $C$ is strongly convex.
\end{prop}
\begin{proof}
    The local convexity of $C$ and the fact that it is connected ensure that $\mathring{C}$ is connected (see for instance \cite[Lemma 1.5]{CG72}). Therefore, we may apply \autoref{teor: supporting} to $\mathring{C}$, so that every point $p\in\partial\mathring{C}$ admits a supporting element. Moreover, since $C$ is contained in a hemisphere, we also have that
    \[
        \mathring{C}\setminus \cut(p)=\mathring{C},
    \]
    which is connected. Therefore, we can apply \autoref{prop: weakchar}, and get that $\mathring{C}$ is weakly convex, and, in particular, as in \autoref{oss: weakinhemi}, strongly convex. Finally, observing that closedness and local convexity ensure $C=\bar{\mathring{C}}$ (see \cite[Theorem 1.6]{CG72}), then $C$ inherits the strong convexity of $\mathring{C}$.

\end{proof}
The following theorem is due to Bangert in \cite[Theorem 1]{B81}.
\begin{teor}[Monotonicity of perimeter]
\label{teor: monotonicity}
    Let $M$ be a Riemannian manifold, and let $C_1,C_2\subseteq M$ such that $C_1$ is totally convex in $C_2$, and $\mathring{C_1}\ne\emptyset$. Assume moreover that $C_2$ has strongly Lipschitz boundary, and $\abs{C_2\setminus C_1}<+\infty$. Then
    \[
        \Hn(\partial C_1)\le \Hn(\partial C_2).
    \]
\end{teor}
The proof of this theorem in the Euclidean case only relies on proving that the projection onto the convex set $C_1$ is a $1$-Lipschitz function (see for instance \cite[Proposition 5.3]{B11}), while the Riemannian case requires a different proof. Even if the monotonicity theorem requires some regularity on the external set $C_2$, we can still prove that this is not restrictive in the case in which $C_2$ is locally convex. Indeed, we have the following result due to Walter in \cite[Theorem 6.1]{W76}.
\begin{teor}
\label{teor: regularity}
    Let $M$ be a Riemannian manifold, and let $C\subset M$ be a closed, locally convex set. Then $C$ has strongly Lipschitz boundary.
\end{teor}
Joining \autoref{teor: monotonicity} and \autoref{teor: regularity}, we get:
\begin{cor}
\label{cor: monotonicity}
     Let $C_1,C_2\subseteq \Sp^n$ be two closed strongly convex sets such that $\mathring{C_1}\ne\emptyset$. If $C_1\subseteq C_2$, then
    \[
        \Hn(\partial C_1)\le \Hn(\partial C_2).
    \]
\end{cor}
\begin{proof}
    It is sufficient to notice that, as in \autoref{oss: strongtototal}, $C_1$ is totally convex in $C_2$. Indeed, \autoref{teor: regularity} ensures the strongly Lipschitz regularity of the boundary, and \autoref{teor: monotonicity} applies.
\end{proof}
We now give some definitions of convexity of continuous functions on Riemannian manifolds, see for instance \cite[\S1]{GW76} for a reference on the topic.

\begin{defi}
    Let $M$ be a Riemannian manifold, and let $f:M\to \R$ be a continuous function. We say that:
    \begin{enumerate}[label=(\alph*)]
    \item $f$ is \emph{convex} if for every geodesic $\gamma:[a,b]\to M$ we have $f\circ \gamma$ is convex on $[a,b]$;
    \item $f$ is \emph{strictly convex} if for every $p\in M$ and for every convex function $\varphi\in C^\infty(M)$ there exists an $\eps>0$ such that $f-\eps \varphi$ is convex in a small neighbourhood of $p$.
    \end{enumerate}
\end{defi}
These definitions are related to the geometry of the sublevel sets.
\begin{prop}
\label{prop: sublevelconvex}
    Let $M$ be a Riemannian manifold, and let $f:M\to \R$ be a continuous function. Then:
    \begin{enumerate}[label=(\roman*)]
        \item if $f$ is convex, then for every $t\in\R$ the set $\Set{x\in M | f(x)<t}$ is totally convex in $M$;
        \item assume that $f$ is strictly convex, and $M$ is weakly convex; for every $t\in\R$, if the set $\Set{x\in M | f(x)<t}$ is compact, then it is locally strictly convex.
    \end{enumerate}
    \end{prop}
\begin{proof}
We only show \textit{(i)}, and we refer to \cite[Lemma 2.4]{B78} for the proof of \textit{(ii)} (note that the assumption on the weak convexity of $M$ ensures the connectedness of the sublevel set of $f$). Let $\gamma:[a,b]\to M$ be a geodesic, and assume that
\[
f(\gamma(a))<t \qquad \qquad f(\gamma(b))<t.
\]
Then, by the definition of convexity, for every $\alpha\in[0,1]$,
\[
f\bigl(\gamma\bigl(a+\alpha(b-a)\bigr)\bigr)\le(1-\alpha)f\bigl(\gamma(a)\bigr)+\alpha f\bigl(\gamma(b)\bigr)<t,
\]
and the assertion is proved.
\end{proof}
We aim to inspect the geometric properties of inner parallel and outer parallel of convex sets. Cheeger and Gromoll, in \cite[Theorem 1.10]{CG72} proved that, for a given convex set $C$ in a Riemannian manifold $M$ with positive sectional curvatures, the distance function
\[
\rho(x)= -d(x,\partial C)
\]
is convex in $\mathring{C}$. This implies that the inner parallel sets $C_t$ are totally convex in $\mathring{C}$. However, we will need some more refined results that can be found in \cite[Theorem 2.1, Theorem 2.3]{B78}, and we summarize in the following. Let $C\subset M$, and denote by $\rho=\rho_C$ the signed distance function 
\[
\rho(x)=\begin{cases}
-d(x,\partial C) &\text{if }x\in\mathring{C},\\
d(x,C) &\text{if }x\notin\mathring{C}.
\end{cases}
\]
Then we have 
\begin{teor}
\label{teor: convexdist}
    Let $M$ be a Riemannian manifold, and let $C$ be a connected, compact, locally convex set. Then the following hold:
    \begin{enumerate}[label=(\roman*)]
        \item if $C$ is locally strictly convex, then there exists $\delta>0$ such that the function
        \[
        \rho+\frac{1}{2}\rho^2
        \]
        is strictly convex on $\mathring{C}^\delta\setminus C$;
        \item if the sectional curvatures on $C$ are negative, then there exists $\delta>0$ such that the function
        \[
        \rho+\frac{1}{2}\rho^2
        \]
        is strictly convex on $\mathring{C}^\delta\setminus C$;
        \item if the sectional curvatures on $C$ are positive, then the function
        \[
            \rho-\log(-\rho)
        \]
        is strictly convex on $\mathring{C}$.
    \end{enumerate}
\end{teor}
\begin{oss}
Despite \cite[Theorem 2.1]{B78} only proves \textit{(i)}, result \textit{(ii)} directly follows from the same proof using a negative upper bound on the sectional curvatures to conclude (see also the proof of \cite[Corollary 2.6]{B78}).
\end{oss}
\begin{cor}
\label{cor: convexpar}
    Let $C\subset \Sp^n$ be a closed strongly convex set. Then:
    \begin{enumerate}[label=(\roman*)]
        \item if $C$ is strongly convex and locally strictly convex, then for small $\delta>0$ we have that the outer parallel sets $(C)^t$ are strongly convex and locally strictly convex for every $t<\delta$;
        \item the inner parallel sets $(C)_t$ are locally strictly convex and strongly convex for every $t>0$.
    \end{enumerate}
\end{cor}
\begin{proof}
    By the condition \textit{(i)} in \autoref{teor: convexdist} we get that $(C)^t$, for small values of $t$ is locally convex. Indeed, for every interior point $p$ of $(C)^t$ it is sufficient to observe that a small strongly convex ball contained in $(C)^t$ always exists. If $p\in\partial(C)^t$, since we can find a small strongly convex ball $B$ contained in $(\mathring{C})^\delta\setminus C$, then the convexity of the function $\rho+\rho^2/2$ ensures that $B\cap (C)^t$ is strongly convex. 

    Moreover, $C$ is connected and contained in a hemisphere, as already seen in \autoref{oss: convexinhemi}. Therefore, $(C)^t$, for small $t$, is connected and contained in the same hemisphere, which implies, by \autoref{prop: localtostrong}, that $(C)^t$ is strongly convex. Finally, by \autoref{prop: sublevelconvex}, we also get that for small $t$ the set $(C)^t$ is locally strictly convex.

    Let us now study the inner parallels $(C)_t$. Analogously to the case of the outer parallels, condition \textit{(iii)} in \autoref{teor: convexdist} yields that the inner parallel sets $(C)_t$ are locally strictly convex. Moreover, the convexity of the function $\rho-\log(-\rho)$ ensures that the sets $(C)_t$ are totally convex in $\mathring{C}$ (see \textit{(i)} in \autoref{prop: sublevelconvex}). Since $C$ is strongly convex, then the total convexity of $(C)_t$ in $\mathring{C}$ gives that the inner parallel sets $(C)_t$ are strongly convex.
\end{proof}
Now we state an approximation theorem proved by Bangert in \cite[Theorem 2.2, Corollary 2.5, Corollary 2.6]{B78}.
\begin{teor}
\label{teor: approx}
    Let $M$ be a Riemannian manifold, and let $C\subset M$ be a connected, compact, locally convex set such that $\mathring{C}\ne \emptyset$. Moreover, assume that either:
    \begin{enumerate}[label=(\alph*)]
        \item $C$ is locally strictly convex;
        \item the sectional curvatures are positive on $C$;
        \item the sectional curvatures are negative on $C$;
    \end{enumerate}
    then there exists a sequence of connected, compact, locally convex sets $C_k$ with $C^\infty$ boundaries such that
    \[
        \lim_{k\to+\infty} d^H(C_k,C)+d^H(\partial C_k,\partial C)=0.
    \]
\end{teor}

\begin{oss}
    Results \textit{(b)} and \textit{(c)} are a direct consequence of \textit{(a)} and \autoref{teor: convexdist} joint with \autoref{prop: sublevelconvex}: in the case \textit{(b)} one approximates the inner parallel sets, while in the case \textit{(c)} one approximates the outer parallel sets.
    \end{oss}

\begin{cor}
\label{cor: approx}
    Let $C\subset\Sp^n$ be a closed strongly convex set such that $\mathring{C}\ne \emptyset$. Then there exists a sequence of closed strongly convex sets $C_k$ with $C^{\infty}$ boundaries such that
    \[
        \lim_{k\to+\infty} d^H(C_k,C)+d^H(\partial C_k,\partial C)=0.
    \]
\end{cor}
\begin{proof}
    Since the sectional curvatures in $\Sp^n$ are positive, \autoref{teor: approx} applies and we find an approximating sequence of connected, compact, locally convex sets $C_k$ with $C^\infty$ boundaries and such that $\mathring{C}_k\ne \emptyset$. Therefore, the Hausdorff convergence also allows us to assume that $C_k$ are contained in the same hemisphere in which $C$ is contained. By \autoref{prop: localtostrong}, we get that $C_k$ are strongly convex.
\end{proof}
\subsection{Curvature measures}
\label{curvatures}
In this section we define curvature measures introduced in $\R^n$ by Federer in \cite{F59} and explicitly computed by Zähle in \cite{Z86}, while successively extended to simply connected space forms by Kohlmann in \cite{K91}.

\medskip

As a first step, we define sets of positive reach. Given a Riemannian manifold $M$, for every $p\in M$ and for every $r>0$ we denote by $B_r(p)$ the metric ball centered in $p$ of radius $r>0$. For small enough $r>0$ we have that $B_r(p)$ coincides with the geodesic ball $\exp_p(B_r(0))$. 
\begin{defi}
    Let $M$ be a Riemannian manifold, and let $\Omega\subset M$ be a non-empty set. For every $p\in M$ we call \emph{metric projection} of $p$ onto $\Omega$ any $q\in \bar\Omega$ such that
    \[
        d(p,q)=d(p,\Omega).
    \]
    When it is unique we write $q=\sigma_\Omega(p)$.
\end{defi}
\begin{defi}
    Let $M$ be a Riemannian manifold, and let $\Omega\subset M$ be a non-empty set. For every $q\in \Omega$ we define the \emph{reach} of $q$ with respect to $\Omega$ as
    \[
        \mathcal{R}(q)=\sup\Set{r>0 | \forall p\in B_r(q) \text{ there exists a unique metric projection of $p$ onto $\Omega$}};
    \]
    we define the \emph{reach} of $\Omega$ as 
    \[
        \mathcal{R}(\Omega)=\inf_{q\in\Omega}\mathcal{R}(q);
    \]
    we say that $\Omega$ is \emph{of positive reach} if $\mathcal{R}(\Omega)>0$.
\end{defi}
In $\R^n$ we have that every convex set $C$ is a set of positive reach with
\[
\mathcal{R}(C)=+\infty.
\]
 On simply connected space forms similar but slightly different results hold for connected, compact, locally convex sets. First, we state a result due to Walter in \cite[Theorem 1]{W74}.
\begin{prop}
    Let $M$ be a Riemannian manifold, and let $C\subseteq M$ be a closed locally convex set. Then $C$ is of positive reach.
\end{prop}
In the specific case of simply connected space forms we have the following result that can be found in \cite[Lemma 2.2]{K94}.
\begin{prop}
\label{prop: lowboundreach}
    Let $M$ be a simply connected space form of curvature $\ka$, and let $C\subseteq M$ be a connected, compact, locally convex set. Then:
    \begin{enumerate}[label=(\roman*)]
        \item if $\ka<0$, then 
        \[
        \mathcal{R}(C)=+\infty;
        \]
        \item if $\ka>0$, then 
        \[
        \mathcal{R}(C)\ge \frac{\pi}{2}\frac{1}{\sqrt{\ka}}.
        \]
    \end{enumerate}
    \end{prop}
\begin{defi}
    Let $(M,g)$ be a Riemannian manifold, let $\Omega\subseteq M$ be a set of positive reach, and let $U$ be an open neighborhood of $\Omega$ such that the metric projection $\sigma_\Omega$ is well defined for every $p\in U$. We denote by 
    \[
        \nu_\Omega(p)=\frac{\exp^{-1}_{\sigma_\Omega(p)}(p)}{\norma{\exp^{-1}_{\sigma_\Omega(p)}(p)}_g} \in TM,
    \]
    and we define the \emph{unit normal bundle}
    \[
        \mathcal{N}(\Omega)=\nu_\Omega(U\setminus\Omega).
    \]
\end{defi}
Since we are going to integrate over $\mathcal{N}(\Omega)$, we need some regularity property of the normal bundle, that is proved in \cite[Theorem 4.3]{W76}. In the following, we are equipping $TM$ with the canonical Sasaki metric.
\begin{teor}
    Let $M$ be a Riemannian $n$-manifold, and let $\Omega\subseteq M$ be a set of positive reach. Then $\mathcal{N}(\Omega)$ is a strongly Lipschitz $(n-1)$-submanifold of $TM$. Moreover, if $\Omega$ is compact, then $\mathcal{N}(\Omega)$ is compact and there exists $\eta=\eta(\Omega)>0$ such that for every $0<r<\eta$, we have that
    \[
    \restr{\nu_\Omega}{\partial(\Omega)^r}: \partial(\Omega)^r \to \mathcal{N}(\Omega)
    \]
    is a locally bi-Lipschitz homeomorphism.
    \end{teor}
    
    \begin{defi}
\label{defi: curvnormbundl}
Let $M$ be a Riemannian $n$-manifold, let $\Sigma$ be a $C^2$ oriented, compact, embedded $(n-1)$-submanifold of $M$ of positive reach, and let $v\in\mathcal{N}(\Sigma)$. Let us denote by
\[
\Pi: TM \to M
\]
the canonical projection such that $\Pi(p,\xi)=p$ for every $(p,\xi)\in TM$. For every $1\le j\le n-1$ we denote by $k_j(v):=k_j(\Pi(v))$ the principal curvatures of $\Sigma$ in $\Pi(v)$, and by
\[
H_j(v)=\sum_{1\le i_1<\dots<i_j\le n-1} k_{i_1}(v)\dots k_{i_j}(v)
\]
the $j$-th homogeneous symmetric form of the principal curvatures. We also denote by 
\[
H_0(v)=1.
\]
\end{defi}

\begin{defi}
        Let $M$ be a complete Riemannian $n$-manifold of constant sectional curvature $\ka$. We define the functions 
        \[
            \sn_\ka(t)=\begin{dcases}
                \frac{1}{\sqrt{-\ka}}\sinh(\sqrt{-\ka}t) &\text{if }\ka<0, \\
                t &\text{if }\ka=0 \\
                \frac{1}{\sqrt{\ka}}\sin(\sqrt{\ka}t) &\text{if }\ka>0, \\
            \end{dcases}
        \]
        and $\cn_\ka$=$\sn_\ka'$. We also let for $1\le  j\le n$
        \[
        L_j(t):=\int_0^t \cn_\ka^{n-j}(t)\sn_\ka^{j-1}(t)\,dt,
        \]
        and
        \[
        L_0(t)=1.
        \]
    \end{defi}
The following theorem is due to Kohlmann in \cite{K91}, but we also point out that Allendoerfer in \cite{A48} proves a Steiner formula for regular convex sets in spheres with different techniques.
    \begin{teor}[Steiner formula on simply connected space forms]
    \label{teor: steiner}
        Let $M$ be a simply connected space form of dimension $n$ and curvature $\ka$, and let $\Omega\subset M$ be a set of positive reach. Let $U$ be an open set in which the metric projection $\sigma_\Omega$ is well defined. For every $j=0,\dots,n$ there exist Radon measures $\Phi_j(\Omega\,;\cdot)$ on $U$ such that the following hold: if $E\subset M$ is a bounded Borel set, and $s>0$ is such that 
        \[
            \sigma_\Omega^{-1}(E)\cap\overline{(\Omega)^s}\subset U,
        \]
        then we have
        \[\Hn(\sigma_\Omega^{-1}(E)\cap\partial(\Omega)^s)=\sum_{r=0}^{n-1} \cn_\ka^{r}(s)\sn_\ka^{n-1-r}(s) \Phi_r(\Omega;E),\]
and
        \[
            \abs{\sigma_\Omega^{-1}(E)\cap(\Omega)^s}=\sum_{r=0}^n L_{n-r}(s)\Phi_{r}(\Omega;E).
        \]
        In particular, 
        \[
        \Phi_n(\Omega\,; E)=\abs{\Omega\cap E} 
        \]
        Moreover, if $\partial\Omega$ is a $C^2$ compact, embedded $(n-1)$-submanifold of $M$, then for every bounded Borel set $E\subset M$ and for every $r=0,\dots, n-1$ we have
        \begin{equation}
        \label{eq: curvaturemeasure}
        \Phi_r(\Omega; E)=\int_{\partial\Omega\cap E} H_{n-1-r}(p)\,d\Hn(p).
        \end{equation}
    \end{teor}
    In the following, we denote by $\Phi_r(\Omega)$ the total measure, namely
    \[
        \Phi_r(\Omega)=\Phi_r(\Omega; M).
    \]
    \begin{oss}
        Notice that for sets $\Omega$ such that $\partial\Omega$ is of class $C^2$ the definition of the \emph{curvature measures} $\Phi_r$ given in \cite[Theorem 2.7]{K91} is equivalent to \eqref{eq: curvaturemeasure}. Indeed, let $\Omega$ be such that $\partial\Omega$ is a $C^2$ compact, embedded $(n-1)$-submanifold of $M$. Using the explicit definition of the measures $\Phi_r$ in \cite{K91}, we have
        \[
        \Phi_r(\Omega;E)=\int_{\mathcal{N}(\Omega)\cap \Pi^{-1}(E)}H_{n-1-r}(v)\prod_{i=1}^{n-1}\frac{1}{\sqrt{1+k_i(v)^2}}d\mathcal{H}^{n-1}(v),
        \]
        where $H_j(v)$ and $k_i(v)$ are defined in \autoref{defi: curvnormbundl}. We start by noticing that the regularity on $\partial\Omega$ ensures that  
        \[
        \restr{\nu_\Omega}{\partial\Omega} :\partial\Omega \to TM
        \]
        is a $C^1$ map, and using Area Formula (\autoref{teor: area}) with the change of variables $\nu_\Omega(p)=v$, we have
        \[
            \Phi_r(\Omega; E)=\int_{\partial\Omega\cap E}H_{n-1-r}(p)\prod_{i=1}^{n-1}\frac{1}{\sqrt{1+k_i(p)^2}} \, J\nu_\Omega\,d\Hn.
        \]
        In particular, as Kohlmann computed in \cite[Equation (2.6) for $\eps=0$]{K91} (to help the reader compare the following equation with Kohlmann's, we recall that: $j_1=\sn_\ka$ and $j_2=\cn_\ka$, while the functions $f_i$ are defined in \cite[Equation (1.20)]{K91}),
        \[
        J\nu_\Omega=\prod_{i=1}^{n-1}\sqrt{1+k_i(p)^2},
        \]
        and we have \eqref{eq: curvaturemeasure}.
        \end{oss}

\begin{oss}
\label{oss: perimeter}
    Let $g$ denote the metric on the simply connected space form $M$. In the following, we explicit the dependence on the metric. Notice that if $\Omega$ is a set of positive reach with $\partial\Omega$ strongly Lipschitz, then we have that for every open set $E$
    \begin{equation}
    \label{eq: Minkowski=hausdorff}
        \Phi_{n-1}^g(\Omega;E)=\Hn_g(\partial\Omega\cap E).
    \end{equation}
    Indeed, since the Steiner formula holds, then for every open set $E$ we have
    \[
        \Phi_{n-1}^g(\Omega;E)=\lim_{s\to 0^+}\frac{\abs{(\Omega^s\setminus \Omega)\cap E}_g}{s},
    \]
    which is the definition of (relative) Minkowski perimeter (or $(n-1)$-dimensional Minkowski content). If $M=\R^n$ equipped with the Euclidean metric, then the equality \eqref{eq: Minkowski=hausdorff} is classical (see for instance \cite[Theorem 2.106]{AFP00}). If $M$ is a generic simply connected space form, then it is possible to obtain \eqref{eq: Minkowski=hausdorff} from the Euclidean case using normal coordinates. Let $p_0\in\partial\Omega$ and let $\eps>0$, there exists a $\delta=\delta(p_0,\eps)>0$ such that we can define an exponential normal chart mapping onto $\mathcal{U}_\eps=\exp_{p_0}^{-1}(B_\delta(p_0))$ such that the metric $g$ in coordinates is given by $g_{ij}=\delta_{ij}+O(\eps)$. In particular, if $g_e$ denotes the Euclidean metric on $\R^n$, the diffeomorphism
    \[
        \id: (\mathcal{U}_\eps, g) \to (\mathcal{U}_\eps,g_e)
    \]
    is a bi-Lipschitz function with
    \[
        \Lip(\id)\le 1+\eps \qquad \Lip(\id^{-1})\le 1+\eps.
    \]
    Therefore, if we denote by $\Hn_g$ the $(n-1)$-Hausdorff measure with respect to the metric $g$, by $\abs{\cdot}_g$ the Riemannian volume, by $\Hn_e$ the $(n-1)$-Hausdorff measure with respect to the Euclidean metric $g_e$, and by $\abs{\cdot}_e$ the Lebesgue measure on $\R^n$, then we get for every Borel set $A\subset\mathcal{U}_\eps$ the following estimates (up to changing $\eps$)
    \begin{equation}    
    \label{eq: estimates}
    \begin{gathered}
        (1+\eps)^{-1}\Hn_g(A)\le \Hn_e(A)\le (1+\eps)\Hn_g(A), \\[9 pt]
        (1+\eps)^{-1}\abs{A}_g\le \abs{A}_e\le (1+\eps)\abs{A}_g, \\[9 pt]
        \Set{x\in \mathcal{U}_\eps | d_g(x,A)< (1+\eps)^{-1}}\subset \Set{x\in \mathcal{U}_\eps | d_e(x,A)< s}\subset \Set{x\in \mathcal{U}_\eps | d_g(x,A)< (1+\eps)s}.
    \end{gathered}
    \end{equation}
    Using the estimates \eqref{eq: estimates} and the fact that the equality \eqref{eq: Minkowski=hausdorff} holds on $\R^n$ for $\Hn_{e}$, then we get (up to choosing a smaller $\eps$) for every $r<\delta$,
    \begin{equation}
    \label{eq: localminkhaus}
        (1-\eps)\Hn_g(\partial\Omega\cap B_r(p_0))\le \Phi_{n-1}^g(\Omega;B_r(p_0))\le (1+\eps)\Hn_g(\partial\Omega\cap B_r(p_0)),
    \end{equation}
    where $B_r(p_0)$ denotes the Euclidean ball of radius $r$ centered in $p_0$.
    Equation \eqref{eq: localminkhaus} in particular implies that the measure $\Phi_{n-1}(\Omega;\cdot)$ is absolutely continuous with respect to $\restr{\Hn_g}{\partial\Omega}$, and that there exists a $\Hn_g$-measurable density $\rho$ such that
    \[
        \Phi_{n-1}(\Omega;E)=\int_{\partial\Omega\cap E} \rho\,d\Hn_g
    \]
    with
    \[
        1-\eps\le \rho\le 1+\eps.
    \]
    Sending $\eps$ to 0 we get $\rho=1$ and \eqref{eq: Minkowski=hausdorff}.
\end{oss}    
        
     \begin{oss} Notice that in the case $\ka = 0$ we get the Steiner polynomial
 \[
            \abs{\sigma_\Omega^{-1}(E)\cap(\Omega)^s}=\abs{\Omega\cap E}+\sum_{k=1}^{n}\frac{s^k}{k} \Phi_{n-k}(\Omega;E).
        \]
 \end{oss}

We now give a continuity property for the curvature measures, and we refer to \cite[Theorem 2.4]{K94} for the proof.
\begin{teor}
\label{teor: curvconv}
    Let $M$ be a simply connected space form of dimension $n$ and curvature $\ka$. Let $\Omega_k\subset M$ be a sequence of compact sets with non-empty boundaries. Let us assume that for some $\delta>0$ and for some compact set $\Omega\subset M$ we have
    \[
        \mathcal{R}(\Omega_k)\ge \delta \qquad \qquad \lim_k d^H(\Omega_k,\Omega)=0.
    \]
    Then for every $r=0,\dots,n$
    \[
        \Phi_r(\Omega_k;\cdot)\xrightharpoonup{\hspace{15pt}} \Phi_r(\Omega;\cdot)
    \]
    in the sense of Radon measures.
\end{teor}

Finally, we state an Alexandrov-Fenchel inequality on the sphere comparing the curvature measure $\Phi_{n-2}$ with $\Phi_{n-1}$, and we refer to \cite[Theorem 1.5]{MS16} for the proof of the regular case. 
\begin{teor}
\label{teor: smoothalexfench}
    Let $\Omega\subset\Sp^n$ be a closed strongly convex set with $C^2$ boundary. Then 
    \[
        \left(\frac{\Phi_{n-2}(\Omega)}{(n-1)\sigma_n}\right)^2\ge \left(\frac{\Phi_{n-1}(\Omega)}{\sigma_n}\right)^{\frac{2(n-2)}{n-1}}-\left(\frac{\Phi_{n-1}(\Omega)}{\sigma_n}\right)^2,
    \]
    and the equality holds if and only if $\Omega$ is a geodesic ball. 
\end{teor}
\begin{cor}
\label{cor: alexfench}
    Let $\Omega\subset\Sp^n$ be a closed strongly convex set. Then 
    \[
        \left(\frac{\Phi_{n-2}(\Omega)}{(n-1)\sigma_n}\right)^2\ge \left(\frac{\Phi_{n-1}(\Omega)}{\sigma_n}\right)^{\frac{2(n-2)}{n-1}}-\left(\frac{\Phi_{n-1}(\Omega)}{\sigma_n}\right)^2.
    \]
\end{cor}
\begin{proof}
    If $\partial\Omega$ is of class $C^2$, then the result follows from \autoref{teor: smoothalexfench}.
    
    For the general case, let $\Omega$ be a closed strongly convex set. By \autoref{cor: approx}, we can find closed strongly convex sets $\Omega_k$ with smooth boundaries such that
    \[
        \lim_k d^H(\Omega_k,\Omega)=0.
    \]
    In particular, we have
    \begin{equation}
        \label{eq: approxalexfench}
          \left(\frac{\Phi_{n-2}(\Omega_k)}{(n-1)\sigma_n}\right)^2\ge \left(\frac{\Phi_{n-1}(\Omega_k)}{\sigma_n}\right)^{\frac{2(n-2)}{n-1}}-\left(\frac{\Phi_{n-1}(\Omega_k)}{\sigma_n}\right)^2.
    \end{equation}
    Since $\Omega_k$ are strongly convex, we have by \autoref{prop: lowboundreach}
    \[
        \mathcal{R}(\Omega_k)\ge \frac{\pi}{2}.
    \]
    Therefore, we can apply \autoref{teor: curvconv}, and passing to the limit in \eqref{eq: approxalexfench}, the assertion follows.
\end{proof}
\subsection{Isoperimetric inequality}
\begin{defi}[Minkowski Perimiter]
Let $\Omega\subset\Sp^n$ be a Borel set. We define the \emph{lower Minkowski content} as
\[
\Mink_-(\Omega):=\liminf_{s\to0^+}\dfrac{\abs{\Omega^s}-\abs{\Omega}}{s}.
\]
\end{defi}
We now state the isoperimetric inequality on spheres in terms of Minkowski content. The original proof of this result is due to Schmidt in \cite{S43} (see also \cite[Theorem 3.15, Theorem 1.52, Theorem 5.18]{R23}).
\begin{teor}   
 Let $\Omega\subset \Sp^n$ be a measurable set, and let $B\subset \Sp^n$ be a geodesic ball having the same measure as $\Omega$. Then,
 \[
 \Mink_-(\partial B)\le \Mink_-(\partial \Omega),
 \]
 and the equality holds if and only if $\Omega$ is a geodesic ball. 
\end{teor}
In particular, \autoref{oss: perimeter} ensures the following corollary for strongly convex sets of the sphere.

\begin{cor}
\label{teor: isop}
    Let $\Omega\subset \Sp^n$ be an open set such that $\bar\Omega$ is strongly convex. Let $B\subset \Sp^n$ be a geodesic ball having the same measure as $\Omega$, then
     \[
         \Hn(\partial B)\le \Hn(\partial \Omega),
     \]
 and the equality holds if and only if $\Omega$ is a geodesic ball.
\end{cor}

\section{Proof of the main theorem}
\label{proof1}
In this section, for strongly convex sets $\Omega\subset\Sp^n$, we denote by $P(\Omega)=\Hn(\partial\Omega)$.

Let $\Omega\subset \Sp^n$ be an open set with strongly Lipschitz boundary, then the variational characterization \eqref{robin} is well posed and the minimizers are weak solutions of \eqref{robinproblem}. 
\begin{oss}\label{radiality}
    Let $D$ be a geodesic ball of center $q$ and radius $R>0$ in the sphere $\Sp^n$. We recall that the eigenfunctions relative to the first eigenvalue $\lambda_\beta(D)$ are all proportional. Therefore, by the rotational symmetry of $D$ and the rotational invariance of the equation \eqref{robinproblem}, we have that all the first eigenfunctions on $D$ are radial. Precisely, $u(p)=\psi(d(p,q))$ for some function $\psi$ solution to the one-dimensional problem
\[\begin{cases}
    \psi''+(n-1)\cot(r)\psi'+\lambda_\beta(D)\psi=0 &r\in(0,R),\\[3 pt]
    \psi'(0)=0,\\[3 pt]
    \psi'(R)+\beta\psi(R)=0.
\end{cases}\] 
Moreover letting $\phi(\rho)=\psi(R-\rho)$ we can write $u$ as a function of the distance from the boundary of the ball $D$, indeed for every $p\in D$
\[d(p,\partial D)=R-d(p,q),\]
so that $u(p)=\phi(d(p,\partial D))$.
\end{oss}

For every $\Omega\subset\Sp^n$, we denote by $R_\Omega$ its inradius, that is 
\[
R_\Omega=\max_{p\in\Omega}d(p,\partial\Omega).
\]
We have the following
\begin{lemma}\label{lemma:estonderper}
Let $\Omega\subset\Sp^n$ be a closed, strongly convex set, and let $\Omega_t=(\Omega)_t$. Then for almost every $t\in(0,R_{\Omega})$ the function $P(\Omega_t)$ is differentiable and
\begin{equation}
\label{odeper}-\dfrac{d}{dt}P(\Omega_t)\ge (n-1)\left(\sigma_n^{\frac{2}{n-1}} P(\Omega_t)^{\frac{2(n-2)}{n-1}}-P(\Omega_t)^2\right)^{\frac{1}{2}}.
\end{equation}
\end{lemma}
\begin{proof}
From the strong convexity of $\Omega$, by \autoref{cor: convexpar} we have that the for every $R_\Omega>s>t>0$ the inner parallel sets $\Omega_t$ and $\Omega_s$ are strongly convex, so that \autoref{cor: monotonicity} ensures that
\[
P(\Omega_s)\le P(\Omega_t).
\]
In particular, the function $t\mapsto P(\Omega_t)$ is monotonic decreasing and hence it is differentiable almost everywhere. Fix $t\in(0,R_\Omega)$, for every $s>0$ sufficiently small, by \autoref{cor: convexpar}, we have that the sets $(\Omega_t)^s$ are strongly convex. Moreover, since by definition
\[
(\Omega_t)^s\subseteq\Omega_{t-s},
\]
and both are strongly convex, we can apply \autoref{cor: monotonicity} again, so that 
\[
P((\Omega_t)^s)\le P(\Omega_{t-s}).
\] 
In particular, we get for almost every $t\in(0,R_\Omega)$
\[
-\dfrac{d}{dt}P(\Omega_t)=\lim_{s\to0^+}\dfrac{P(\Omega_{t-s})-P(\Omega_t)}{s}\ge\lim_{s\to0^+}\dfrac{P((\Omega_{t})^s)-P(\Omega_t)}{s}=\restr{\dfrac{d}{ds}P((\Omega_t)^s)}{s=0}.
\]
By the Steiner formula (\autoref{teor: steiner}), we have
\[
\restr{\dfrac{d}{ds}P((\Omega_t)^s)}{s=0}=\Phi_{n-2}(\Omega_t).
\]
Hence, \eqref{odeper} follows from the Alexandrov-Fenchel inequality \autoref{cor: alexfench}.
\end{proof}

In order to prove \autoref{teorema1} we need a comparison result that relates $P((\Omega)_t)$ and $P((B)_t)$. 
\begin{lemma}
\label{lem: comparison}
    Let $f:[a,b]\to \R$ be a monotone decreasing function, and let $g:[a,b]\to \R$ be an absolutely continuous function. Assume that there exists a Lipschitz function $F:\R\to\R$ such that
    \[
        \begin{dcases}
            f(a)\le g(a),\\
            f'(t)\le F(f(t)) &\text{for a.e. }t\in(a,b), \\
            g'(t) = F(g(t))  &\text{for a.e. }t\in(a,b),
        \end{dcases}
    \]
    then $f(t)\le g(t)$ for every $t\in[a,b]$.
\end{lemma} 
\begin{proof}
    First we recall that since $f$ is decreasing then for every $t,s\in(a,b]$ such that $t<s$, (see for instance \cite[Corollary 3.29]{AFP00})
    \begin{equation}
    \label{eq: fder}
        f(s)-f(t)\le f(s^-)-f(t^+)\le \int_t^s f'(\rho)\,d\rho\le\int_t^s F(f(\rho))\,d\rho,
    \end{equation}
    where we used the notation $f(s^\pm)=\lim_{\eps\to0^\pm}f(s+\eps)$. On the other hand, for $g$ we have the equality
    \begin{equation}
    \label{eq: gder}
        g(s)-g(t)= \int_t^s g'(\rho)\,d\rho=\int_t^s F(g(\rho))\, d\rho.
    \end{equation}
    Subtracting \eqref{eq: gder} to \eqref{eq: fder}, and letting $w(t)=f(t)-g(t)$, then
    \begin{equation}
    \label{eq: wder}
    \begin{split}
        w(s)-w(t)&\le \int_t^s \left(F(f(\rho))-F(g(\rho))\right)\, dt \\[7 pt]
        &\le L\int_t^s\abs{w(\rho)}\,d\rho,
    \end{split}
    \end{equation}
    where $L$ is the Lipschitz constant of $F$. We also notice that by the monotonicity of $f$ and the continuity of $g$ we have
    \begin{equation}
    \label{eq: limitinequalities1}
        w(s)\le w(s^-)  \qquad\qquad \forall s\in(a,b],
    \end{equation}
    \begin{equation}
    \label{eq: limitinequalities2}
        w(s^+)\le w(s) \qquad \qquad \forall s\in[a,b).
    \end{equation}
    By contradiction, let us assume that for some $t_0\in(a,b]$ we have $w(t_0)>0$. Then \eqref{eq: limitinequalities1} ensures that for a suitable $\delta>0$ and for every $s\in(t_0-\delta,t_0]$ we have $w(s)>0$. Let 
    \[
        \tau = \sup \set{t \in[a,t_0) | w(t)\le 0},
    \]
    so that $a\le \tau\le t_0-\delta$, and 
    \begin{equation}
    \label{eq: wspos}
        w(s)>0 \qquad \forall s\in (\tau,t_0].
    \end{equation}
   By definition of $\tau$, we have
    \begin{equation}
    \label{eq: wpos}
    w(\tau^+)\ge 0.
    \end{equation}
    We claim that $w(\tau)=0$. Indeed, if $\tau=a$, then the initial condition yields
    \[
        w(a)\le 0,
    \]
    and by \eqref{eq: wpos}, joint with \eqref{eq: limitinequalities2}, we get $0\le w(a^+)\le w(a)\le 0$. If $\tau>a$, then, by definition of $\tau$,
    \[
       w(\tau^-)\le 0. 
    \]
    Using  \eqref{eq: wpos} joint with \eqref{eq: limitinequalities1} and \eqref{eq: limitinequalities2}, we also have $0\le w(\tau^+)\le w(\tau^-)\le 0$, and the claim is proved. 
    
    Therefore, since $w(\tau)=0$, \eqref{eq: wder} reads as follows: for every $s\in(\tau,t_0)$ we have
    \[
        w(s)\le L\int_\tau^s w(\rho)\,d\rho.
    \]
    By the integral form of the Gronwall inequality (see for instance \cite[Lemma 3.2]{HH17}, which is a particular case of \cite[Theorem 3.1]{H99})
    , we get $w\le 0$ in $[\tau,t_0]$, which is in contradiction with \eqref{eq: wspos}. 
\end{proof}

We are now able to prove the main theorem.

\begin{proof}[Proof of \autoref{teorema1}]
   Let $D$ be a strongly convex geodesic ball such that 
\[
P(\Omega)=P(D),
\] 
and let $R$ be its radius. The isoperimetric inequality (\autoref{teor: isop}) and the fact that both $D$ and $\Omega$ are contained in a hemisphere, ensure that $\abs{\Omega}\le \abs{D}$. Since $R_\Omega$ is the radius of the biggest ball contained in $\Omega$, we also obtain $R_\Omega\le R$, and the equality holds if and only if $\Omega$ is a ball of radius $R_\Omega$.

For every $t\in(0,R_\Omega)$, let 
\[
\Omega_t=(\Omega)_t, \qquad \text{and} \qquad D_t=(D)_t
\]
be the inner parallel sets of $\Omega$ and $D$ respectively. Then from \autoref{lemma:estonderper} we have that
\[
\dfrac{d}{dt}P(\Omega_t)\le- (n-1)\left(\sigma_n^{\frac{2}{n-1}} P(\Omega_t)^{\frac{2(n-2)}{n-1}}-P(\Omega_t)^2\right)^{\frac{1}{2}},
\]
while, by direct computation, the same estimate holds for the perimeter of $D_t$  with the equality sign 
\[
\dfrac{d}{dt}P(D_t)=- (n-1)\left(\sigma_n^{\frac{2}{n-1}} P(D_t)^{\frac{2(n-2)}{n-1}}-P(D_t)^2\right)^{\frac{1}{2}}.
\]
The comparison lemma (\autoref{lem: comparison}) ensures that for every $t\in(0,R_\Omega)$
\begin{equation}
\label{eq: isoplevels}
 P(\Omega_t)\le P(D_t).
\end{equation}

Let $u$ be an eigenfunction on $D$ and let $\phi\colon[0,R]\to\R$ be as in \autoref{radiality}, then for every $p\in D$ 
 \[u(p)=\phi(d(p,\partial D)).\]
For every $p\in\Omega$, let us define
\[v(p)=\phi(d(p,\partial\Omega)),\]
so that $v\in H^1(\Omega)$, and 
\[\lambda_\beta(\Omega)\le \dfrac{\displaystyle\int_\Omega\abs{\nabla v}^2\,d\mu+\beta\int_{\partial\Omega} v^2\,d\Hn}{\displaystyle\int_\Omega v^2\,d\mu}.\] 
By direct computation, we have that
\[\int_{\partial\Omega} v^2\,d\Hn=\phi^2(0)P(\Omega)=\int_{\partial D} u^2\,d\Hn.\]
While, using coarea formula (\autoref{teor: coarea}) with $f(p)=d(p,\partial\Omega)$ and \eqref{eq: isoplevels}, we have 
\begin{equation}
\label{eq: L2comp}
\int_\Omega v^2\,d\mu=\int_0^{R_\Omega} \phi^2(t)P(\Omega_t)\,dt\le\int_0^{R_\Omega} \phi^2(t)P(D_t)\,dt\le \int_D u^2\,d\mu,
\end{equation}
and
\[\int_\Omega \abs{\nabla v}^2\,d\mu=\int_0^{R_\Omega} (\phi'(t))^2 P(\Omega_t)\,dt\le\int_0^{R_\Omega} (\phi'(t))^2 P(D_t)\,dt\le \int_D \abs{\nabla u}^2\,d\mu.\]
Then 
\[\int_\Omega \abs{\nabla v}^2\,d\mu+\beta\int_{\partial\Omega} v^2\,d\Hn\le\int_D \abs{\nabla u}^2+\beta\int_{\partial D} u^2\,d\Hn<0.\]
Hence,
\[\lambda_\beta(\Omega)\le \dfrac{\displaystyle\int_\Omega\abs{\nabla v}^2\,d\mu+\beta\int_{\partial\Omega} v^2\,d\Hn}{\displaystyle\int_\Omega v^2\,d\mu}\le \dfrac{\displaystyle\int_D\abs{\nabla u}^2\,d\mu+\beta\int_{\partial D} u^2\,d\Hn}{\displaystyle\int_D u^2\,d\mu}=\lambda_\beta(D).\]

Finally, if the equality $\lambda_\beta(\Omega)=\lambda_\beta(D)$ holds, then the equality in \eqref{eq: L2comp} gives that $R_\Omega=R$, which implies that $\Omega$ is a geodesic ball of radius $R_\Omega$.
\end{proof}
Following the approach of \cite{AGM22} we now prove \autoref{teorema2}.
\begin{proof}[Proof of \autoref{teorema2}]
    As in the proof of \autoref{teorema1}, let $\phi\colon[0,R]\to\R$ be such that $u(p)=\phi(d(p,\partial D))$ and let $v(p)=\phi(d(p,\partial\Omega))$. In order to obtain \eqref{quant} we can better estimate the $L^2$-norm of the test $v$. Indeed, 
\[\begin{split}\int_\Omega v^2\,d\mu=&\int_0^{R_\Omega}\phi^2(t)P(\Omega_t)\,dt
\le\int_0^{R_\Omega}\phi^2(t)P(D_t)\,dt\\[15 pt]
=&\int_0^R\phi^2(t)P(D_t)\,dt-\int_{R_\Omega}^R \phi^2(t)P(D_t)\,dt
\le\int_{D}u^2\,d\mu-u_m^2\int_{R_\Omega}^R P(D_t)\,dt\\[15 pt]
\le&\int_\Omega u^2\,d\mu-u_m^2(\abs{D}-\abs{\Omega})
=\int_\Omega u^2\,d\mu\left(1-\dfrac{u_m^2}{\norma{u}_{L^2(D)}^2}(\abs{D}-\abs{\Omega})\right),
\end{split}\]
where we have used that, by definition of inradius, for a suitable ball $B_{R_\Omega}$ of radius $R_\Omega$ we have $B_{R_\Omega}\subseteq\Omega$. Therefore, computations analogous to the ones done in \autoref{teorema1} lead to
\[\begin{split}\lambda_\beta(\Omega)&\le \dfrac{\displaystyle\int_\Omega\abs{\nabla v}^2\,d\mu+\beta\int_{\partial\Omega} v^2\,d\Hn}{\displaystyle\int_\Omega v^2\,d\mu}\\[15 pt]&\le \dfrac{\displaystyle\int_D\abs{\nabla u}^2\,d\mu+\beta\int_{\partial D} u^2\,d\Hn}{\displaystyle\int_D u^2\,d\mu\left(1-\dfrac{u_m^2}{\norma{u}_{L^2(D)}^2}(\abs{D}-\abs{\Omega})\right)}\\[15 pt]&=\lambda_\beta(D)\left(1-\dfrac{u_m^2}{\norma{u}_{L^2(D)}^2}(\abs{D}-\abs{\Omega})\right)^{-1}.\end{split}\]
So that, reordering the terms, \eqref{quant} is proved.
\end{proof}
\section{Further remarks}
\label{remarks}
In this section, we show that the same arguments used to prove \autoref{teorema1} cannot be used for strongly convex sets in the hyperbolic setting. Even though it is possible to generalize the Alexandrov-Fenchel inequalities to the hyperbolic space under strict convexity assumptions (see for instance \cite[Theorem 1.1]{WX14}), the main difficulty here is to extend \autoref{cor: convexpar} to the hyperbolic space $\mathbb{H}^n$. In particular, we can construct convex sets for which the inner parallel sets are not convex. To show an example, let us fix some notation. Let
\[
\norma{x}_e=\sqrt{\scalar{x}{x}}
\]
be the Euclidean norm, and let $\mathbb{H}^n$ be represented in the Poincaré half-space model:
\begin{gather*}
    \mathbb{H}^n=\Set{(\hat{x},x_{n})\in\R^{n} | x_{n}>0}, \\[7 pt]
    g_x(v,w)=\frac{\scalar{v}{w}}{x_{n}^2}, \\[7 pt]
    d(x,y)=2\arcsinh\left(\frac{\norma{x-y}_e}{2\sqrt{x_{n}y_{n}}}\right).
\end{gather*}
We also recall the shape of the geodesics in $\mathbb{H}^n$: let $p,q\in\mathbb{H}^n$, if $\hat{p}=\hat{q}$, then the geodesic $\gamma_{pq}$ connecting the two is the vertical line passing through $p$ and $q$; if $\hat{p}\ne \hat{q}$, then the geodesic $\gamma_{pq}$ connecting $p$ and $q$ is the unique circular arc touching orthogonally the plane $\{x_n=0\}$. We could describe the circular arc (non-parametrized by arc length) as
\begin{equation}
\begin{aligned}
\label{eq: geodesic}
\gamma_{pq}: [t_p,t_q] &\longrightarrow \mathbb{H}^n \\
                 t &\longmapsto x_0+R\left(t\hat{w},\sqrt{1-t^2}\right),
\end{aligned}
\end{equation}
where $x_0=(\hat{x}_0,0)$ is the center of the circular arc, $R=\norma{x_0-p}_e$ is the radius of the arc,
\[
\hat{w}=\frac{\hat{p}-\hat{q}}{\norma{\hat{p}-\hat{q}}_e},
\]
and
\begin{gather*}
[t_p,t_q]\subset (-1,1), \\
\gamma_{pq}(t_p)=p \qquad \qquad \gamma_{pq}(t_q)=q.
\end{gather*}

We now divide the construction into three simple steps.

\noindent
\textbf{Step 1:}
Consider the cylinder
\[
C=\Set{(\hat{x},x_n)\in\R^n | \begin{gathered}
    \norma{\hat{x}}_e\le 1
\end{gathered}}.
\]
C is convex:

consider $p,q\in C$, and let $\gamma_{pq}$ be the geodesic connecting the points. If $\hat{p}=\hat{q}$, then we can represent the geodesic (non-parametrized by arc length) as 
\[
\gamma_{pq}(t)=(\hat{p},t),
\]
and obviously $\gamma_{pq}(t)\in C$ for every $t$. If $\hat{p}\ne\hat{q}$, then for the geodesic $\gamma_{pq}=(\hat{\gamma}_{pq},\gamma_{pq}^n)$ we have that $\hat{\gamma}_{pq}([t_p,t_q])$ is the segment joining $\hat{p}$ and $\hat{q}$ in $\R^{n-1}$, so that
\[
\norma{\hat{\gamma}_{pq}(t)}_e\le \max\{\norma{\hat{p}}_e,\norma{\hat{q}}_e\}\le 1,
\]
for every $t\in[t_p,t_q]$.

\noindent
\textbf{Step 2.} For any fixed vertical line $r(\hat{x}_0)=\set{(\hat{x},x_n) |\hat{x}=\hat{x}_0}$, the level sets of the distance from $r(\hat{x}_0)$ are cones: 

it is sufficient to notice that for every point $(\hat{x}_0,x_n)$ the geodesics orthogonal to $r(\hat{x}_0)$ in $(\hat{x}_0,x_n)$ are all contained in the hemisphere of radius $x_n$ centered in $x_0=(\hat{x}_0,0)$, so that, with a direct computation,
\begin{equation*}
(r(\hat{x}_0))^t=\Set{x\in\R^n | d(x,r_0)\le t}=\Set{(\hat{x},x_n)\in\R^n | \norma{\hat{x}-\hat{x}_0}_e\le \sinh(t)\, x_n}.
\end{equation*}

\noindent

\noindent
\textbf{Step 3:} the inner parallel sets $(C)_\delta$ are not convex for every choice of $\delta>0$.

Indeed, notice that
\[
(C)_\delta=C\cap \bigcap_{\norma{\hat{x}_0}_e=1}\overline{\mathbb{H}^n\setminus(r(\hat{x}_0))^\delta}=\Set{(\hat{x},x_n)\in\R^n | \begin{gathered}
    x_n>0, \\
    \norma{\hat{x}}_e\le 1-\sinh(\delta)\,x_n.
\end{gathered}}
\]
If for instance we take $q\in (C)_\delta$ such that $\norma{\hat{q}}_e=1-\sinh(\delta)q_n$ and $p=(0,1/\sinh(\delta)))$, then the minimal geodesic $\gamma_{pq}=(\hat{\gamma},\gamma^n)$ connecting $p$ and $q$ lies outside the cone $(C)_\delta$: we can write, for $t\in[t_p,t_q]$,
\[
\gamma(t)=x_0+R\left(t\hat{w},\sqrt{1-t^2}\right)
\]
as defined in \eqref{eq: geodesic}; notice that by concavity
\begin{equation}
\label{eq: gamman}
\begin{split}
\gamma^n(t)&>\gamma^n(t_p)+\frac{t-t_p}{t_q-t_p}(\gamma^n(t_q)-\gamma^n(t_p))  \\[7 pt]
&=p_n+\frac{t-t_p}{t_q-t_p}(q_n-p_n)\qquad\qquad  \forall t\in(t_p,t_q),
\end{split}
\end{equation}
and that, since $\hat{p}=0$, then $\hat{w}$ and $\hat{x}_0$ are proportional, then $\norma{\hat{\gamma}(t)}_e$ is linear in $t$, so that it is of the form
\begin{equation}
\label{eq: gammahat}
\norma{\hat{\gamma}_{pq}(t)}_e=\norma{\hat{p}}_e+\frac{t-t_q}{t_p-t_q}\norma{\hat{q}}_e.
\end{equation}
Using the fact that $p,q\in\partial (C)_\delta$, and in particular \eqref{eq: gamman} and \eqref{eq: gammahat}, which implies that the geodesic is not contained in $(C)_\delta$.

\subsubsection*{Acknowledgements} 
The five authors are members of Gruppo Nazionale per l’Analisi Matematica, la Probabilità e le loro Applicazioni
(GNAMPA) of Istituto Nazionale di Alta Matematica (INdAM). 

 The author Cristina Trombetti has been supported by the Project MiUR PRIN-PNRR 2022: "Linear and Nonlinear PDE’S: New directions and Applications", P2022YFAJH

The authors Paolo Acampora, Emanuele Cristoforoni and Carlo Nitsch were partially supported by the project GNAMPA 2023: "Symmetry and asymmetry in PDEs", CUP\_E53C22001930001. 

The authors Paolo Acampora and Emanuele Cristoforoni were partially supported by the 2024 project GNAMPA 2024: "Modelli PDE-ODE nonlineari e proprieta' di PDE su domini standard e non-standard", CUP\_E53C23001670001. 

The author Antonio Celentano was partially supported by the INdAM-GNAMPA project 2023 "Modelli matematici di EDP per fluidi e strutture e proprieta' geometriche delle soluzioni di EDP", cod. CUP\_E53C22001930001, and by the INdAM-GNAMPA project 2024 "Problemi frazionari: proprietà quantitative ottimali, simmetria, regolarità", cod. CUP\_E53C23001670001.

We would like to thank Marco Pozzetta for his suggestions and fruitful discussions about the topic.
 \newpage 
 
\printbibliography[heading=bibintoc]
\Addresses
\end{document}